\documentclass[bj,preprint]{imsart}

\RequirePackage[OT1]{fontenc}
\RequirePackage{amsthm,amsmath}
\RequirePackage[colorlinks,citecolor=blue,urlcolor=blue]{hyperref}
\RequirePackage[numbers]{natbib}
\usepackage{amsmath}
\usepackage{amssymb}
\usepackage{amsthm}
\usepackage{color}
\usepackage{amsfonts}
\usepackage{bbm}

\arxiv{arXiv:1311.6923}

\startlocaldefs

\newcommand{\mmp}{\mathbb{P}}

\newcommand{\od}{\overset{\mathrm{d}}{=}}
\newcommand{\dod}{\overset{\mathrm{d}}{\to}}

\newcommand{\vaguec}{\overset{v}{\to}}

\newcommand{\me}{\mathbb{E}}
\newcommand{\mz}{\mathbb{Z}}
\newcommand{\mr}{\mathbb{R}}
\newcommand{\mn}{\mathbb{N}}

\newcommand{\Q}{\mathbb{Q}}

\newcommand{\lin}{\underset{n\to\infty}{\lim}}

\newcommand{\lit}{\underset{t\to\infty}{\lim}}

\newcommand{\1}{\mathbbm{1}}

\newtheorem{thm}{Theorem}[section]
\newtheorem{lemma}[thm]{Lemma}

\newtheorem{assertion}[thm]{Proposition}
\theoremstyle{definition}

\newtheorem{example}[thm]{Example}
\theoremstyle{remark}
\newtheorem{rem}[thm]{Remark}

\endlocaldefs

\begin{document}

\begin{frontmatter}

\title{Asymptotics of random processes with immigration II: convergence to stationarity}

\runtitle{Asymptotics of random processes with immigration II}

\begin{aug}

\author{\fnms{Alexander} \snm{Iksanov}\thanksref{a,e1}\corref{}\ead[label=e1,mark]{iksan@univ.kiev.ua}}
\author{\fnms{Alexander} \snm{Marynych}\thanksref{a,e2}\ead[label=e2,mark]{marynych@unicyb.kiev.ua}}
\and
\author{\fnms{Matthias} \snm{Meiners}\thanksref{b,e3}\ead[label=e3,mark]{meiners@mathematik.tu-darmstadt.de}}

\runauthor{A. Iksanov et al.}

\affiliation{Taras Shevchenko National University of Kyiv and Technical University of Darmstadt}
\address[a]{Faculty of Cybernetics, Taras Shevchenko National University of Kyiv, 01601 Kyiv, Ukraine. \printead{e1,e2}}
\address[b]{Fachbereich Mathematik, Technische Universitat Darmstadt, 64289 Darmstadt, Germany. \printead{e3}}

\end{aug}

\begin{abstract}
Let $X_1, X_2,\ldots$ be random elements of the
Skorokhod space $D(\mr)$ and $\xi_1, \xi_2, \ldots$ positive
random variables such that the pairs $(X_1,\xi_1),
(X_2,\xi_2),\ldots$ are independent and identically distributed.
We call the random process $(Y(t))_{t \in \mr}$ defined by
$Y(t):=\sum_{k \geq
0}X_{k+1}(t-\xi_1-\ldots-\xi_k)\1_{\{\xi_1+\ldots+\xi_k\leq t\}}$,
$t\in\mr$ random process with immigration at the epochs of a
renewal process. Assuming that $X_k$ and $\xi_k$ are independent
and that the distribution of $\xi_1$ is nonlattice and has finite
mean we investigate weak convergence of $(Y(t))_{t\in\mr}$ as
$t\to\infty$ in $D(\mr)$ endowed with the $J_1$-topology. The
limits are stationary processes with immigration.
\end{abstract}




\end{frontmatter}

\section{Introduction}
Denote by $D(\mr)$ the Skorokhod space of right-continuous
real-valued functions which are defined on $\mr$ and have finite
limits from the left at each point of $\mr$. Let
$X:=(X(t))_{t\in\mr}$ be a random process with paths in $D(\mr)$
satisfying $X(t)=0$ for all $t < 0$ and let $\xi$ be a positive
random variable. 
Further, let $(X_1,\xi_1), (X_2, \xi_2),\ldots$ be i.i.d.~copies of the pair $(X,\xi)$
and denote by $(S_n)_{n\in\mn_0}$ the zero-delayed random walk with increments $\xi_k$, that is,
\begin{equation*}
S_0 := 0, \qquad    S_n:=\xi_1+\ldots+\xi_n, \quad  n\in\mn.
\end{equation*}
Following \cite{Iksanov+Marynych+Meiners:2015I} we call {\it
random process with immigration} the process $Y :=
(Y(t))_{t\in\mr}$ defined by
\begin{equation}    \label{eq:Y(t)}
Y(t) := \sum_{k\geq 0}X_{k+1}(t-S_k) 
, \quad   t\in\mr.
\end{equation}
The motivation for the term is discussed in
\cite{Iksanov+Marynych+Meiners:2015I} where the reader can also
find a list of possible applications and some bibliographic
comments.

Continuing the line of research initiated in
\cite{Iksanov:2013,Iksanov+Marynych+Meiners:2014,Iksanov+Marynych+Meiners:2015I}
we are interested in weak convergence of random processes with
immigration. We treat the situation when $\me [|X(t)|]$ is finite
and tends to $0$ quickly as $t\to\infty$ while $\me \xi<\infty$.
Then $Y$ is the superposition of a regular stream of freshly
started processes with quickly fading contributions of the
processes that started early. As $t$ becomes large, these
competing effects balance on a distributional level and $Y$
approaches stationarity. Under these assumptions the joint
distribution of $(X,\xi)$ should affect the asymptotic behavior of
$Y$. However, we refrain from investigating this by assuming that
$X$ and $\xi$ are independent.

We are only aware of two papers which are concerned with weak
convergence of processes $Y$ to their stationary versions in the
case when $\xi$ has distribution other than
exponential\footnote{Various references related to a simpler
situation when the distribution of $\xi$ is exponential are given
in \cite{Iksanov+Marynych+Meiners:2015I}.}. In
\cite{Iksanov+Marynych+Meiners:2014} the authors prove weak
convergence of the finite-dimensional distributions of
$\big(\sum_{k\geq 0} h(ut-S_k)\1_{\{S_k \leq ut\}}\big)_{u>0}$ as
$t\to\infty$ for a deterministic function $h$. Notice that the
time is assumed scaled ($ut$) in Theorem 2.1 of
\cite{Iksanov+Marynych+Meiners:2014}, whereas we translate the
time ($u+t$) in Theorem \ref{main1} of the present paper.
Furthermore, the approach taken in
\cite{Iksanov+Marynych+Meiners:2014}, which is partly discussed in
Remark \ref{remark1}, differs from that exploited here. Theorem
6.1 in \cite{Miller:1974} is a result about weak convergence of
the one-dimensional distributions of $Y(t)$ as $t\to\infty$ for
$X(t)=g(t,\eta)$, where $g:[0,\infty)\times \mr\to [0,\infty)$ is
a deterministic function satisfying certain conditions and $\eta$
is a random variable independent of $\xi$. As has already been
mentioned in \cite{Iksanov+Marynych+Meiners:2014}, the cited
theorem does not hold in the generality stated in
\cite{Miller:1974}.

\section{Main results}
Before we formulate our results, some preliminary work has to be
done.
\subsection{Stationary renewal processes and stationary random processes with immigration}
 Suppose that $\mu := \me \xi<\infty$, and that
the distribution of $\xi$ is nonlattice, i.e., it is not
concentrated on any lattice $d\mz$, $d>0$. Further, we stipulate
hereafter that the basic probability space on which $(X_k)_{k \in
\mn}$ and $(\xi_k)_{k \in \mn}$ are defined is rich enough to
accommodate
\begin{itemize}
    \item
        an independent copy $(\xi_{-k})_{k \in \mn}$ of $(\xi_k)_{k \in \mn}$;
    \item
        a random variable $\xi_0$ which is independent of $(\xi_k)_{k \in \mz \setminus \{0\}}$ and has distribution
        \begin{equation*}
        \mmp\{\xi_0 \in {\rm d}x\} = \mu^{-1} \me [\xi \1_{\{\xi \in {\rm d}x\}}], \quad    x\geq 0;
        \end{equation*}
   \item
        a random variable $U$ which is independent of $(\xi_k)_{k\in\mz}$ and has the uniform distribution on $[0,1]$;
   \item
        a family $(X_k)_{k\in\mz}$ of i.i.d.~random elements of $D(\mr)$ that is independent of $(\xi_k)_{k\in\mz}$ and $U$.
\end{itemize}
Set $\nu(t):=\inf\{k\in\mn_0: S_k>t\}$, $t\in\mr$,
\begin{equation*}
S_{-k} := -(\xi_{-1}+\ldots+\xi_{-k}), \ \ k \in \mn,
\end{equation*}
and
\begin{equation*}
S^\ast_0 := U \xi_0,\ S^\ast_{-1} := -(1-U) \xi_0, \;
S^\ast_k = S^\ast_0+S_k, \; k \in \mn, \;  S^\ast_{-k} :=
S^\ast_{-1} + S_{-k+1}, \; k\in \mn\setminus\{1\}.
\end{equation*}
Recall\footnote{See e.g.\ \cite[Section 3.10]{Resnick:1992}.} that
the distribution of both, $S^\ast_0$ and $-S^\ast_{-1}$, coincides
with the limiting distribution of the overshoot $S_{\nu(t)}-t$ and
the undershoot $t-S_{\nu(t)-1}$ as $t\to\infty$:
\begin{equation*}
\mmp\{S^\ast_0\in {\rm d}x\} = \mmp\{-S^\ast_{-1}\in {\rm d}x\} =
\mu^{-1}\mmp\{\xi>x\}\1_{(0,\,\infty)}(x){\rm d}x.
\end{equation*}
It is well-known \cite[Chapter 8, Theorem 4.1]{Thorisson:2000}
that the point process $\sum_{k \in \mz} \delta_{S^\ast_k}$ is
shift-invariant, that is, $\sum_{k \in \mz} \delta_{S^\ast_k}$ has
the same distribution as $\sum_{k \in \mz} \delta_{S^\ast_k+t}$
for every $t \in \mr$. In particular, the intensity measure of
this process is a constant multiple of the Lebesgue measure where
the constant can be identified as $\mu^{-1}$ by the elementary
renewal theorem. In conclusion,
\begin{equation}    \label{eq:intensity=Lebesgue}
\me\bigg[\sum_{k \in \mz} \delta_{S^\ast_k}({\rm d}x) \bigg]    ~=~ \frac{{\rm d}x}{\mu}.
\end{equation}
Fix any $u\in\mr$.
Since $\lim_{k \to -\infty} S_k^\ast=-\infty$ a.s., the sum
\begin{equation*}
\sum_{k\leq -1}X_{k+1}(u+S_k^\ast)\1_{\{S_k^\ast\geq -u\}}
\end{equation*}
is a.s.~finite because the number of non-zero summands is a.s.~finite. Define
\begin{equation*}
Y^\ast(u):=\sum_{k\in\mz}X_{k+1}(u+S_k^\ast) = \sum_{k\in\mz}X_{k+1}(u+S_k^\ast)\1_{\{S_k^\ast \geq -u\}}
\end{equation*}
with the random variable $Y^\ast(u)$ being a.s.~finite provided
that the series $\sum_{k\geq 0}X_{k+1}(u+S_k^\ast)\1_{\{S_k\geq
-u\}}$ converges in probability. It is natural to call
$(Y^{\ast}(u))_{u\in\mr}$ {\it the stationary random process with
immigration.}

\subsection{Convergence in $D(\mr)$}
Consider the subset $D_0$ of the Skorokhod space $D(\mr)$ composed
of those functions $f\in D(\mr)$ which have finite limits
$f(-\infty):=\lim_{t\to-\infty}f(t)$ and
$f(\infty):=\lim_{t\to+\infty}f(t)$. For $a,b\in\mr$, $a<b$ let
$d_0^{a,b}$ be the Skorokhod metric on $D[a,b]$, i.e.,
$$  d_0^{a,b}(x,y)
= \inf_{\lambda \in \Lambda_{a,b}} \bigg( \sup_{t \in [a,b]} |x(\lambda(t)) - y(t)| \vee \sup_{s \not = t} \Big| \log \Big(\frac{\lambda(t)-\lambda(s)}{t-s}\Big)\Big|\bigg)$$
where $\Lambda_{a,b} = \{\lambda: \lambda \text{ is a strictly increasing continuous function on } [a,b] \text{ with } \lambda(a)=a, \lambda(b)=b\}.$
Following \cite[Section 3]{Lindvall:1973}, for
$f,g\in D_0$, put
\begin{equation*}
d_0(f,g):=d_0^{0,1}(\overline{\phi}(f),\overline{\phi}(g)),
\end{equation*}
where $$\phi(t):=\log (t/(1-t)),\;\;t\in
(0,1),\;\;\phi(0)=-\infty,\;\;\phi(1):=+\infty$$ and
$$
\overline{\phi}:D_0 \to D[0,1],\;\;\overline{\phi}(x)(\cdot):=x(\phi(\cdot)),\;\;x\in D_0.
$$
Then $(D_0,d_0)$ is a complete separable metric space. Mimicking
the argument given in Section 4 in \cite{Lindvall:1973} and using
$d_0$ as a basis one can construct a metric $d$ (its explicit form
is of no importance here) on $D(\mr)$ such that $(D(\mr),d)$ is a complete
separable metric space. We shall need the following
characterization of the convergence in $(D(\mr),d)$, see Theorem 1(b)
in \cite{Lindvall:1973} and Theorem 12.9.3(ii) in
\cite{Whitt:2002} for the convergence in $D[0,\infty)$.
\begin{assertion}\label{conv_in_d}
Suppose $f_n, f\in D(\mr)$, $n\in\mn$. The following conditions are
equivalent:
\begin{itemize}
\item[(i)] $f_n\to f$ in $(D(\mr),d)$ as $n\to\infty$;
\item[(ii)] there exist 
\begin{eqnarray*}
&&\hspace{-1cm}\lambda_n \in \Lambda:=\{\lambda : \lambda\text{ is a strictly increasing  continuous function on } \mr \\
&&\hspace{8cm}\text{ with } \lambda(\pm\infty)=\pm\infty\}
\end{eqnarray*}
such that, for any finite $a$ and $b$, $a<b$,
$$ \lin \max\Big\{\sup_{u\in [a,\,b]}
|f_n(\lambda_n(u))-f(u)|,\,\sup_{u\in [a,\,b]}|\lambda_n(u)-u|\Big\}=0;
$$
\item[(iii)] for any finite $a$ and $b$, $a<b$ which are continuity points of $f$
it holds that $f_n|_{[a,\,b]}\to f|_{[a,\,b]}$ in
$(D[a,b],d_0^{a,b})$ as $n\to\infty$, where $g|_{[a,\,b]}$ denotes
the restriction of $g\in D(\mr)$ to $[a,b]$.
\end{itemize}
\end{assertion}

\subsection{Main result}

Let $\mathcal{D}_X := \{t \geq 0: \mmp\{X(t)\neq X(t-)\}>0\}$ and
$\Delta_X := \{a-b:a,b\in\mathcal{D}_X\}$. Note that
$\mathcal{D}_X$, and hence $\Delta_X$, may be empty. In the
following, we write `$Z_t\Rightarrow Z$ as $t\to\infty$ on
$(S,d^\ast)$'\, to denote weak convergence of processes on a
complete separable metric space $(S,d^\ast)$ and
`$\overset{d}{\to}$' to denote convergence in distribution of
random variables or random vectors.

Our main result, Theorem \ref{main1}, provides (a) sufficient
conditions for weak convergence of the finite-dimensional
distributions of $(Y(t+u))_{u\in\mr}$ as $t\to\infty$ and (b) more
restrictive sufficient conditions for weak convergence of the same
processes in $(D(\mr),d)$.
\begin{thm}\label{main1}
Suppose that
\begin{itemize}
    \item $X$ and $\xi$ are independent;
    \item $\mu:=\me\xi<\infty$;
    \item the distribution of $\xi$ is nonlattice.
\end{itemize}

\begin{itemize}
    \item[(a)]
        If the function $G(t):=\me [|X(t)|\wedge 1]$ is
        directly Riemann integrable\footnote{See p.\;232 in \cite{Resnick:2002} for the definition of direct Riemann integrability.} (dRi)
        on $[0,\infty)$, then, for each
        $u\in\mr$, the series $\sum_{k\geq 0}X_{k+1}(u+S_k^\ast)\1_{\{S_k^\ast\geq -u\}}$
        is absolutely convergent with probability one, and, for any $n\in\mn$ and any finite $u_1<u_2<\ldots<u_n$,
        \begin{equation}\label{main_relation}
        \big(Y(t+u_1),\ldots, Y(t+u_n)\big) \ \dod \
        \big(Y^\ast(u_1),\ldots, Y^\ast(u_n)\big), \ \ t\to\infty.
        \end{equation}
    \item[(b)]
        If, for some $\varepsilon>0$, the function
        $H_{\varepsilon}(t) := \me [\sup_{u\in[t,\,t+\varepsilon]} |X(u)|\wedge 1]$ is dRi on $[0,\infty)$, and
        \begin{equation}\label{no_common_disc_ass} \mmp\{S_j\in \Delta_X\}=0
        \end{equation}
        for each $j\in\mn$, then
        \begin{equation}\label{main_relation_func}
        Y(t+u) \ \Rightarrow \ Y^\ast(u), \ \ t\to\infty        \quad   \text{in } (D(\mr),d).
        \end{equation}
\end{itemize}
\end{thm}
\begin{rem}\label{remark1}
Condition \eqref{no_common_disc_ass} needs to be checked only if
the set $\mathcal{D}_X$ contains more than one element, and the
distribution of $\xi$ has a discrete component. Otherwise, it
holds automatically.
\end{rem}
\begin{rem}
Establishing weak convergence of finite-dimensional distributions
followed by checking the tightness is the standard approach to
proving weak convergence in the Skorokhod space. To prove Theorem
\ref{main1} we take another route: the two statements of the
theorem are treated independently, the main technical tool being
the continuous mapping theorem applied to relation
\eqref{point_process_conv_2}. It is known that, for any
$\varepsilon>0$, the sequences $(S_n)_{n\in\mn_0}$ and
$(S^\ast_k)_{k\in\mn_0}$ can be coupled such that they become
$\varepsilon$-close with probability one, see
\cite[pp.\;74--75]{Lindvall:1992}. By exploiting this observation,
it is proved in \cite[Theorem 2.1]{Iksanov+Marynych+Meiners:2014}
that $Y(t)$ converges in distribution to $Y^\ast(0)$ as
$t\to\infty$ for deterministic $X$. Elaborating on the ideas of
the proof of the cited theorem, we could have suggested another
proof of \eqref{main_relation} which would be intuitively more
appealing than the proof given below. However, we have not been
able to overcome the considerable technical obstacles arising when
attempting to prove \eqref{main_relation_func} in this way.
\end{rem}

\section{Discussion of the assumptions of Theorem \ref{main1}}  \label{dri}

Suppose that $X$ is as defined in the introduction, that is,
with probability one, $X$ takes values in $D(\mr)$ and $X(t)=0$ for all
$t<0$. In this subsection, we first derive equivalent conditions for
the direct Riemann integrability of the functions $G(t)=\me
[|X(t)|\wedge 1]$ and
$H_\varepsilon(t)=\me[\sup_{u\in[t,\,t+\varepsilon]}|X(u)|\wedge
1]$ that are more suitable for applications.

With probability one, $X$ takes values in $D(\mr)$, and hence is continuous almost
everywhere (a.e.). This carries over to $t\mapsto
|X(t)|\wedge 1$. Now notice that if $t \mapsto |X(t)|\wedge 1$ is
continuous at $t_0$ and $t_0+\varepsilon$, then $t\mapsto
\sup_{u\in[t,\,t+\varepsilon]} (|X(u)|\wedge 1)$ is continuous at
$t_0$. Consequently, with probability one, the process $t \mapsto
\sup_{u\in[t,\,t+\varepsilon]}(|X(u)|\wedge 1)$ is
a.e.~continuous. This implies that the set of $t_0$ such that $t
\mapsto \sup_{u\in[t,\,t+\varepsilon]}(|X(u)|\wedge 1)$ is
discontinuous at $t_0$ with positive probability has Lebesgue
measure $0$. From Lebesgue's dominated convergence theorem we
conclude that $G$ and $H_{\varepsilon}$ are a.e.~continuous. Since
$G$ and $H_{\varepsilon}$ are also bounded, they must be
locally Riemann integrable. From this we conclude that the direct
Riemann integrability of $G$ is equivalent to
\begin{equation}    \label{dri_mean_eq}
\sum_{k\geq 0} \sup_{t\in[k,\,k+1)} \me [|X(t)|\wedge1]<\infty,
\end{equation}
while the direct Riemann integrability of $H_\varepsilon$ is
equivalent to
\begin{equation}    \label{dri_path_eq_hard}
\sum_{k\geq 0}\sup_{t\in[k,\,k+1)} \me \Big[\sup_{u\in[t,\,t+\varepsilon]} (|X(u)|\wedge 1) \Big] < \infty.
\end{equation}
Moreover, \eqref{dri_path_eq_hard} is equivalent to
\begin{equation}    \label{dri_path_eq}
\sum_{k\geq 0} \me \Big[\sup_{u\in[k,\,k+1)}(|X(u)|\wedge 1)\Big] < \infty
\end{equation}
which particularly implies that $H_{\varepsilon}$ is dRi for every $\varepsilon>0$ whenever it is
dRi for some $\varepsilon>0$.
Indeed,
\begin{eqnarray*}
\sum_{k\geq 0} \me \Big[ \sup_{u\in[k,\,k+1)}(|X(u)|\wedge 1) \Big]
& \leq &
\sum_{k\geq 0}\me \bigg[ \sum_{j=0}^{\lfloor \varepsilon^{-1}\rfloor}\sup_{u\in[k+j\varepsilon,\,k+(j+1)\varepsilon)}(|X(u)|\wedge 1) \bigg]\\
& = & \sum_{j=0}^{\lfloor \varepsilon^{-1}\rfloor} \sum_{k\geq 0}
H_{\varepsilon}(k+j\varepsilon) ~\leq~ \sum_{j=0}^{\lfloor
\varepsilon^{-1}\rfloor} \sum_{k \geq 0} \sup_{t\in[k,\,k+1)}
H_{\varepsilon}(t).
\end{eqnarray*}
Thus  \eqref{dri_path_eq_hard} implies \eqref{dri_path_eq}. To see
that \eqref{dri_path_eq} implies \eqref{dri_path_eq_hard} use
\eqref{eq:sup_k sup_a,b} below with $a=0$ and $b=\varepsilon$.

Here are several cases in which \eqref{dri_mean_eq} and \eqref{dri_path_eq} are equivalent:
\begin{itemize}
    \item[(i)]
        $X(t) \equiv h(t)$ a.s.~for a deterministic function $h$;
    \item[(ii)]
        there exists $t_0>0$ such that, with probability one, $|X(t)|$ is nonincreasing on $[t_0,\infty)$;
    \item[(iii)]
        $\mmp\{|X(t)| \in (0,1)\} = 0$ for all $t \geq 0$,
        $\tau:=\inf\{t\geq 0: X(t)=0\} <\infty$ a.s.~and $X(t)=0$ for all $t \geq \tau$ a.s.~in which case
        \begin{equation}\label{1010}
        \eqref{dri_mean_eq} \Leftrightarrow \eqref{dri_path_eq}
        \Leftrightarrow  \me \tau<\infty.
        \end{equation}
\end{itemize}
Indeed, in case (i), one can omit the expectations in
\eqref{dri_mean_eq} and \eqref{dri_path_eq} since $X$ is
deterministic. The resulting formulae coincide. In case (ii), for
all $k \geq t_0$, $\sup_{t \in [k,\,k+1)} \me[|X(t)| \wedge 1] =
\me[|X(k)| \wedge 1]$ and $\me[\sup_{u \in [k,\,k+1)} |X(u)|
\wedge 1] = \me[|X(k)| \wedge 1]$. Hence the infinite series in
\eqref{dri_mean_eq} and \eqref{dri_path_eq} coincide for all but
finitely many terms. Finally, assume that $X$ satisfies the
assumptions of case (iii). We show that \eqref{1010} holds.

\noindent ``\eqref{dri_mean_eq}\,$\Rightarrow \me \tau<\infty$'':
From the equality
$$ \me [|X(t)|\wedge 1] = \mmp\{|X(t)|\geq 1\} = \mmp\{\tau > t\}$$
we deduce that
$$
\infty > \sum_{k \geq 0}\sup_{t\in[k,\,k+1)} \me \big[|X(t)|\wedge 1\big] \geq \sum_{k\geq 0} \mmp\{\tau > k\} \geq \me \tau.
$$

\noindent ``$\me \tau<\infty \Rightarrow$\,\eqref{dri_path_eq}'':
This implication follows from
\begin{eqnarray*}
\sum_{k\geq 0} \me \Big[\sup_{t\in[k,\,k+1)}(|X(t)|\wedge 1)\Big]
= \me \bigg[\sum_{k=0}^{\lfloor \tau \rfloor}\sup_{t\in[k,\,k+1)}(|X(t)|\wedge 1)\bigg]
\leq \me [\lfloor \tau \rfloor+1]<\infty.
\end{eqnarray*}
Finally the implication
``\eqref{dri_path_eq}\,$\Rightarrow$\,\eqref{dri_mean_eq}'' is obvious.

We now give an example in which $X$ satisfies \eqref{dri_mean_eq},
yet does not satisfy \eqref{dri_path_eq}.
\begin{example}
Let $\eta$ be uniformly distributed on $[0,1]$ and set
$$  X(t):=\sum_{k\geq 1} \1_{\big\{k+{k^2\over k^2+1}\eta \leq t < k+\eta\big\}}, \ \ t\geq 0.  $$
Then inequality \eqref{dri_path_eq} fails to hold,
for $\sup_{t \in [k,\,k+1)} (|X(t)|\wedge 1) = \sup_{t \in [k,\,k+1)} X(t) = 1$ a.s.
On the other hand,
$\sup_{t\in [k,\,k+1)} \me [|X(t)|\wedge 1] = \sup_{t\in [k,\,k+1)} \me [X(t)] = (k^2+1)^{-1}$ for $k\in\mn$, and inequality
\eqref{dri_mean_eq} holds true.
\end{example}

To make the distinction between \eqref{dri_mean_eq} and \eqref{dri_path_eq} more transparent,
we note that \eqref{dri_path_eq} entails the direct Riemann integrability of
$X$ with probability one and, as a consequence, $\lim_{t \to \infty} X(t)=0$ a.s.
On the other hand, the preceding example demonstrates that \eqref{dri_mean_eq}
only guarantees $\lim_{t \to \infty} X(t)=0$ in probability.

We close this subsection with an example in which $Y(t)$ fails to
converge in distribution, as $t\to\infty$.
\begin{example}
Let $X(t)=h(t) := (1 \wedge 1/t^2)\1_{\Q}(t)$, $t \geq 0$, where
$\Q$ denotes the set of rationals. Observe that $G(t)=\me
[|X(t)|\wedge 1]=h(t)$ is Lebesgue integrable but not Riemann
integrable. Let the distribution of $\xi$ be such that $\mmp\{\xi
\in \Q \cap (0,1]\} = 1$ and $\mmp\{\xi = r\}
> 0$ for all $r \in \Q \cap (0,1]$. Then the distribution of $\xi$
is non-lattice. It is clear that $Y(t)=0$ for
$t\in\mr\backslash\Q$. On the other hand, according to Example 2.6
in \cite{Iksanov+Marynych+Meiners:2014}
\begin{equation*} Y(t)  ~\stackrel{\mathrm{d}}{\to}~    \sum_{k \geq 0}
f(S^*_k)    \quad   \text{as } t \to \infty,\ t \in \Q,
\end{equation*}
where $f(t) = 1 \wedge 1/t^2$ for $t \geq 0$. Note that the latter
random variable is positive a.s.
\end{example}

\section{Applications}
Suppose that the first three assumptions of Theorem \ref{main1} hold.

\begin{example}
Let $\mmp\{X(t)\geq 0\}=1$ and $\mmp\{X(t)\in (0,1)\}=0$ for each $t\geq 0$.
Suppose that, with probability one, $X$ gets absorbed into the unique absorbing state
$0$. This means that the random variable $\tau:=\inf\{t: X(t)=0\}$
is a.s.~finite, and $X(t)=0$ for $t \geq \tau$.
Then $\me \tau<\infty$ is necessary and sufficient for \eqref{main_relation} to hold.
Moreover, under the additional assumption \eqref{no_common_disc_ass}, $\me \tau < \infty$ is equivalent to \eqref{main_relation_func}.

Indeed, if $\me\tau<\infty$, then \eqref{1010} ensures that the
functions $G(t)=\me [|X(t)\big|\wedge 1]$ and
$H_{\varepsilon}(t)=\me \big[
\sup_{u\in[t,\,t+\varepsilon]}(|X(u)|\wedge 1)\big]$ (for arbitrary
$\varepsilon>0$) are dRi on $[0,\infty)$. Therefore, \eqref{main_relation} and,
under the additional assumption \eqref{no_common_disc_ass},
\eqref{main_relation_func} follow from Theorem \ref{main1}.

Suppose now that $\me \tau=\infty$. By the strong law of large
numbers, for any $\rho\in (0,\mu)$, there exists an a.s.~finite
random variable $M$ such that $S_k^\ast>(\mu-\rho)k$ for $k\geq
M$. Therefore, for any $u \in \mr$,
\begin{eqnarray*}
\sum_{k\geq 0}\mmp\big\{X_{k+1}(u+S_k^\ast)\1_{\{u+S_k^\ast \geq 0\}}\geq 1|(S_j^\ast)_j\big\}
& = &
\sum_{k\geq 0} \1_{\{u+S_k^\ast\geq 0\}}\mmp\big\{\tau_{k+1}>u+S_k^\ast\big|(S_j^\ast)_j\big\}  \\
& = &
\sum_{k\geq \nu^\ast(-u)}\mmp\big\{\tau_{k+1}-u> S_k^\ast\big|(S_j^\ast)_j\big\}    \\
& \geq &
\sum_{k\geq M \vee \nu^\ast(-u)}\mmp\big\{\tau-u>(\mu-\rho)k\big|(S_j^\ast)_j \big\} \\
& =& \infty \ \ \text{a.s.},
\end{eqnarray*}
where $\tau_k:=\inf\{t: X_k(t)=0\}$, and $\nu^\ast(-u) := \inf\{k \in \mn_0:  S_k^\ast\geq -u\}$.
Given $(S_j^\ast)_j$, the
series $\sum_{k\geq 0}X_{k+1}(u+S_k^\ast)\1_{\{S_k^\ast\geq -u\}}$
does not converge a.s.~by the three series theorem. Since the
general term of the series is nonnegative, then, given
$(S_j^\ast)_j$, this series diverges a.s. Hence, for each
$u\in\mr$, $\sum_{k\in\mz}X_{k+1}(u+S_k^\ast)\1_{\{S_k^\ast\geq
-u\}}=\infty$ a.s., and \eqref{main_relation} cannot hold.

\begin{itemize}
    \item[(a)]
        Let $X$ be a subcritical or critical Bellman-Harris process
         (see Chapter IV in \cite{Athreya+Ney:1972} for the definition and many properties) with a single ancestor,
        and let $\eta$ and $N$ be independent with
        $\eta$ being distributed according to the life length distribution
        and $N$ according to the offspring distribution of the process.
        Suppose that $\mmp\{\eta=0\}=0$, $\mmp\{N=0\}<1$ and
        $\mmp\{N=1\}<1$. Then $Y$ is a subcritical or critical
        Bellman-Harris process with (single) immigration at the epochs of
        a renewal process. $Y$ satisfies \eqref{main_relation} if and only if
        $\me \tau<\infty$. In \cite[Theorem 1]{Pakes+Kaplan:1974}
        the same criterion is derived for the convergence of the one-dimensional distributions via
        an analytic argument. Under the condition $\me \eta<\infty$, which
        entails $\me \tau<\infty$, weak convergence of the one-dimensional
        distributions of a subcritical process with immigration was proved
        in Theorem 3 in \cite{Jagers:1968}. Notice that in the two cited papers multiple immigration is allowed.
        Finally we note that relation
        \eqref{main_relation_func} is equivalent to $\me \tau < \infty$ under the additional condition
        \eqref{no_common_disc_ass}, which holds, for instance, if the
        distribution of $\eta$ is continuous.

    \item[(b)]
        Suppose that $X(t)=\1_{\{\eta>t\}}$ for a
        nonnegative random variable $\eta$. Observe that $X$ is a (degenerate) Bellman-Harris process with $\mmp\{N=0\}=1$.
        Because of its simplicity and its numerous applications
        the corresponding process $Y$ has received considerable attention
        in the literature. We only mention the following interpretations:
        \begin{itemize}
            \item $Y(t)$ is the number of busy servers at time $t$ in the $GI/G/\infty$ queue \cite{Kaplan:1975};
            \item $Y(t)$ is the number of active downloads at time $t$ in a computer network \cite{Konstantopoulos+Lin:98,Mikosch+Resnick:2006};
            \item
                $Y(t)$ is the difference between the number of visits to the segment
                $[0,t]$ of the standard random walk $\big(S_n\big)_{n\in\mn_0}$
                and the perturbed random walk
                $\big(S_n+\eta_{n+1}\big)_{n\in\mn_0}$, where $\big(\eta_n\big)$
                are independent copies of $\eta$ \cite{Alsmeyer+Iksanov+Meiners:2015}.
        \end{itemize}
        In this case $\tau=\eta$. Hence the corresponding $Y$ satisfies
        \eqref{main_relation} and, under the additional assumption
        \eqref{no_common_disc_ass}, \eqref{main_relation_func} if and only if $\me \eta<\infty$.

    \item[(c)]
        Let $X$ be a birth and death process with
        $X(0)=i\in\mn$ a.s. Suppose that $X$ is eventually absorbed at $0$ with probability one.
        Since \eqref{no_common_disc_ass} holds we conclude that the
        corresponding $Y$ satisfies \eqref{main_relation_func} if and only
        if $\me \tau<\infty$. A criterion for the finiteness of $\me \tau$
        expressed in terms of infinitesimal intensities is given in \cite[Theorem 7.1 on p.\;149]{Karlin+Taylor:1975}.
\end{itemize}
\end{example}

\begin{example}
Let $X(t)=\eta f(t)$, where $\eta$ is a random variable
independent of $\xi$, and $f: \mr\to\mr$, with $f(t)=0$ for $t<0$,
belongs to $D(\mr)$.

\begin{itemize}
    \item[(a)]
        Suppose that $\mmp\{\eta=b\}=1$, and that the function $t \to |f(t)|\wedge 1$ is dRi.
        Then relation \eqref{main_relation} and, under the additional assumption \eqref{no_common_disc_ass},
        relation \eqref{main_relation_func} hold.
        Weak convergence of one dimensional distributions was proved in Theorem 2.1 in \cite{Iksanov+Marynych+Meiners:2014}
        under the assumption that the function $t \mapsto |f(t)|$ is dRi, not assuming, however, that $f\in D(\mr)$.\!\footnote{If $f\in D(\mr)$,
        then $f$ is bounded on compact intervals, and the function $t \mapsto |f(t)|\wedge 1$ is dRi if and only if so is $t \mapsto |f(t)|$.}
    \item[(b)]
        Suppose $f(t)=e^{-at}$, $a>0$. If $\me [\log^+|\eta|] < \infty$, then the nonincreasing function
        $$H_\varepsilon(t) = \me \bigg[\underset{u\in [t,\,t+\varepsilon]}{\sup}\, \big(|\eta|e^{-au}\wedge 1\big)\bigg] = \me [|\eta|e^{-at}\wedge 1\big]$$
        is integrable, hence dRi. Further, \eqref{no_common_disc_ass} holds since $X$ is a.s.~continuous.
        Thus Theorem \ref{main1} implies \eqref{main_relation_func}.
        If $\me [\log^+|\eta|] = \infty$, then, by \cite[Theorem 2.1]{Goldie+Maller:2000},
        $$
				\lim_{n \to \infty} \Big|\sum_{k=0}^n \eta_{k+1}\exp(-a S_k^\ast)\Big|=\infty
				$$ in probability
        where $\eta_1,\eta_2,\ldots$ are i.i.d.~copies of $\eta$.
        The latter implies that \eqref{main_relation} and \eqref{main_relation_func} cannot hold.
\end{itemize}
\end{example}

\section{Proof of Theorem \ref{main1}}\label{prmain1}

Let $M_p(\mr)$ be the set of Radon point measures on $\mr$ with
the topology of vague convergence $\vaguec$, and let
$\delta_{x_0}$ denote the probability measure concentrated at
point $x_0 \in \mr$. Recall that, for $m_n, m \in M_p(\mr)$,
\begin{equation*}
m_n \ \vaguec \ m,\;\;n\to\infty
\end{equation*}
if and only if
\begin{equation*}
\lim_{n \to \infty} \int_\mr f(x)m_n({\rm d}x)=\int_{\mr}f(x)m({\rm d}x)
\end{equation*}
for any continuous function $f:\mr\to\mr^+$ with compact support.
According to Proposition 3.17 in \cite{Resnick:1987}, there is a
metric $\rho$ on $M_p(\mr)$ which makes $(M_p(\mr),\rho)$ a
complete separable metric space, while convergence in this metric
is equivalent to vague convergence. Further, for later use,
recall that any $m \in M_p(\mr)$ has a representation of the form
$m = \sum_{k \in \mz} \delta_{t_k}$ for $t_k \in \mr$. Moreover,
this representation is unique subject to the constraints $t_k \leq
t_{k+1}$ for all $k \in \mz$ and $t_{-1} < 0 \leq t_0$. The $t_k$
are given by
\begin{equation}    \label{eq:unique representation of m}
t_k ~=~ \begin{cases}
            \inf\{t \geq 0: m([0,t]) \geq k+1\} &   \text{if } k \geq 0;    \\
            -\inf\{t \geq 0: m([-t,0)) \geq -k\}    &   \text{if } k < 0.
            \end{cases}
\end{equation}

Before we prove Theorem \ref{main1}, we give three auxiliary
lemmas. Lemma \ref{import} given next and the continuous mapping
theorem are the key technical tools in the proof of Theorem
\ref{main1}.
\begin{lemma}   \label{import}
Assume that $\me\xi<\infty$ and that the distribution of $\xi$ is
nonlattice. Then
\begin{equation*}
\sum_{k\geq 0}\delta_{t-S_k} \ \Rightarrow \
\sum_{j\in\mz}\delta_{S^\ast_j},\;\;t\to\infty
\end{equation*}
on $(M_p(\mr), \rho)$.
\end{lemma}
\begin{proof}
Let $h:\mr\to \mr^{+}$ be a continuous function with a compact
support. According to Proposition 3.19 on p.\;153 in \cite{Resnick:1987},
it suffices to prove
$$  \lit \me \bigg[ \exp\bigg(-\sum_{k\geq 0}h(t-S_k)\bigg)\bigg] = \me \bigg[ \exp\bigg(-\sum_{j\in\mz}h(S^{\ast}_j)\bigg)\bigg].  $$
Let $A := \inf \{t: h(t) \not = 0\} > -\infty$ and $g(t) := h(t+A)$, $t \in \mr$.
Then $g(t) = 0$ for $t<0$ and $g$ is dRi on $\mr^+$ as a continuous function with compact support.
Hence, Theorem 2.1 in \cite{Iksanov+Marynych+Meiners:2014} applies and yields
$$  \sum_{k\geq 0} g(t-S_k\big) \dod    \sum_{k \geq 0}g(S^{\ast}_k),\ \ t\to\infty.    $$
This implies convergence of the associated Laplace transforms, hence,
\begin{align*}
\me \bigg[\exp\bigg(-\sum_{k\geq 0} &h(t-S_k)\bigg)\bigg]= 
\me \bigg[\exp\bigg(-\sum_{k\geq 0} g\big(t-A-S_k\big)\bigg)\bigg]  \\
&\to
\me\bigg[\exp\bigg(-\sum_{k \geq 0} g(S^\ast_k)\bigg)\bigg] = \me\bigg[\exp\bigg(-\sum_{k \geq 0} h(S^\ast_k+A)\bigg)\bigg]   \\
&= \me\bigg[\exp\bigg(-\sum_{k \in \mz}
h(S^\ast_k+A)\bigg)\bigg] ~=~ \me\bigg[\exp\bigg(-\sum_{k \in \mz}
h(S^\ast_k)\bigg)\bigg]
\end{align*}
as $t \to \infty$ where the next-to-last equality is due to
the fact that $h(S^\ast_k+A) = 0$ for $k<0$ while the last
equality follows from the distributional shift invariance of
$\sum_{k \in \mz} \delta_{S^\ast_k}$ \cite[Chapter 8, Theorem
4.1]{Thorisson:2000}.
\end{proof}

\begin{lemma}   \label{continuity_lemma1}
Suppose that $t_n\to\ t$ on $\mr$ and $f_n\to f$ in $(D(\mr),d)$, as
$n\to\infty$. Then
$$
f_n(t_n+\cdot)\to f(t+\cdot),\;\;n\to\infty
$$
in $(D(\mr),d)$.
\end{lemma}
\begin{proof}
Without loss of generality we assume that $t=0$. It suffices
to prove that there exist $\lambda_n \in \Lambda$, $n\in\mn$ such that, for any $-\infty<a<b<\infty$,
\begin{equation}    \label{lemma_conv7}
\lim_{n \to \infty} \max\Big\{\sup_{u \in [a,\,b]}|\lambda_n(u)-u|,\sup_{u \in [a,\,b]}|f_n(t_n+\lambda_n(u))-f(u)|\Big\} = 0.
\end{equation}
By assumption, $f_n\to f$ in $(D(\mr),d)$. Hence, there are $\mu_n\in\Lambda$, $n\in\mn$ such that, for any $-\infty<a<b<\infty$,
\begin{equation}\label{lemma_conv3}
\lim_{n \to \infty} \max\Big\{\sup_{u\in [a,\,b]}|\mu_n(u)-u|,\sup_{u \in [a,\,b]}|f_n(\mu_n(u))-f(u)|\Big\} = 0.
\end{equation}
Put $\lambda_n(u):=\mu_n(u)-t_n$ and note that $\lambda_n\in\Lambda$.
Then \eqref{lemma_conv3} can be rewritten as
\begin{equation*}
\lim_{n \to \infty} \max\Big\{\sup_{u\in [a,\,b]}|\lambda_n(u)-u+t_n|,\sup_{u\in
[a,\,b]}|f_n(t_n+\lambda_n(u))-f(u)|\Big\}=0
\end{equation*}
which is equivalent to \eqref{lemma_conv7}, for $\lim_{n \to \infty} t_n=0$.
\end{proof}

\begin{rem}
As was kindly communicated to us by one of the referees the
counterpart of Lemma \ref{continuity_lemma1} with $(D(\mr),d)$ replaced
by $(D[0,\infty), d_1)$, where $d_1$ is the standard $J_1$-metric,
may fail to hold. Take, for instance, $f_n(t): = f(t): =
\1_{[1,\infty)}(t)$ and $t_n: = 1-n^{-1}$. Then $f_n(t_n)=0$ does
not converge to $f(1)=1$ as $n\to\infty$ which implies that
$f_n(t_n+\cdot)$ do not converge to $f(1+\cdot)$ in $(D[0,\infty),
d_1)$.
\end{rem}

Denote by $D(\mr)^{\mz}$ the Cartesian product of countably many copies
of $D(\mr)$ endowed with the topology of componentwise convergence via
the metric
$$  d^\mz((f_k)_{k\in\mz},(g_k)_{k\in\mz}):=\sum_{k\in\mz}2^{-|k|}\big(d(f_k,g_k)\wedge 1\big). $$
Note that $(D(\mr)^{\mz},d^\mz)$ is a complete and separable metric space.
Now consider the metric space $(M_p(\mr)\times D(\mr)^\mz,\rho^{\ast})$
where $\rho^{\ast}(\cdot,\cdot):=d^\mz(\cdot,\cdot)+\rho(\cdot,\cdot)$
(i.e., convergence is defined componentwise). As the Cartesian
product of complete and separable spaces, $(M_p(\mr)\times D(\mr)^\mz,\rho^{\ast})$ is complete and separable.

For fixed $c>0$, $l\in\mn$ and $(u_1,\ldots,u_l)\in\mr^l$, define the mapping
$\phi^{(l)}_c:M_p(\mr)\times D(\mr)^\mz \to \mr^l$ by
\begin{equation*}
\phi^{(l)}_c\big(m,(f_k(\cdot))_{k\in\mz}\big):=\bigg(\sum_{k}f_k(t_k+u_j)\1_{\{|t_k|\leq
c\}}\bigg)_{j=1,\ldots,l}
\end{equation*}
and the mapping $\phi_c:M_p(\mr)\times D(\mr)^\mz \to D(\mr)$ by
\begin{equation*}
\phi_c\big(m,(f_k(\cdot))_{k\in\mz}\big) :=
\sum_{k}f_k(t_k+\cdot)\1_{\{|t_k| \leq c\}}
\end{equation*}
where in the definition of $\phi^{(l)}_c$ and $\phi_c$, the
$t_k$ are given by \eqref{eq:unique representation of m}. It can
be checked that $\phi^{(l)}_c$ and $\phi_c$ are measurable
mappings. For $f\in D(\mr)$, denote by ${\rm Disc}(f)$ the set of
discontinuity points of $f$ on $\mr$.

\begin{lemma}   \label{continuity_lemma2}
The mapping $\phi^{(l)}_c$ is continuous at all points
$(m,(f_k)_{k\in\mz})$ for which $m(\{-c,0,c\})=0$ and for
which $u_1,\ldots,u_l$ are continuity points of $f_k(t_k+\cdot)$
for all $k\in\mz$. $\phi_c$ is continuous at all points
$(m,(f_k)_{k\in\mz})$ satisfying $m(\{-c,0,c\})=0$ and ${\rm
Disc}(f_k(t_k+\cdot))\cap {\rm Disc}(f_j(t_j+\cdot))=\varnothing$
for $k\neq j$.
\end{lemma}
\begin{proof}
Let $c>0$ and suppose that
\begin{equation}\label{lemma_conv4}
\big(m_n,(f^{(n)}_k)_{k\in\mz}\big) \to
\big(m,(f_k)_{k\in\mz}\big),\quad n\to\infty
\end{equation}
on $(M_p(\mr)\times D(\mr)^\mz,\rho^{\ast})$ where
$m(\{-c,0,c\})=0$. Then, in particular, $m_n \vaguec m$ as $n
\to \infty$. Since $m(\{-c,0,c\})=0$, we can apply Theorem
3.13 in \cite{Resnick:1987}, which says that
$m_n([-c,0])=m([-c,0]) =: r_-$ and $m_n([0,c]) = m([0,c]) =: r_+$
for all sufficiently large $n$. For these $n$, with the definition
of $t_k^{(n)}$ and $t_k$ according to \eqref{eq:unique
representation of m}, we have
\begin{align*}
m_n(\cdot \cap [-c,0]) = \sum_{k=1}^{r_-} \delta_{t^{(n)}_{-k}},    \quad & m_n(\cdot \cap [0,c]) = \sum_{k=0}^{r_+-1} \delta_{t^{(n)}_{k}},    \\
m(\cdot \cap [-c,0]) = \sum_{k=1}^{r_-} \delta_{t_{-k}},    \text{ and } & m_n(\cdot \cap [0,c]) = \sum_{k=0}^{r_+-1} \delta_{t_{k}},
\end{align*}
where, of course, the empty sum is understood to be $0$. Theorem
3.13 in \cite{Resnick:1987} further implies that there is
convergence of the points of $m_n$ in $[-c,0]$ to the points of
$m$ in $[-c,0]$ and analogously with $[-c,0]$ replaced by $[0,c]$.
Since $m$ has no point at $0$, this implies that $t_k^{(n)} \to
t_k$ as $n \to \infty$ for $k=-r_-, \ldots, r_+-1$. On the other
hand, \eqref{lemma_conv4} entails $\lim_{n \to \infty}
f^{(n)}_k=f_k$ in $(D(\mr),d)$ for $k=-r_-, \ldots, r_+-1$.
Therefore, Lemma \ref{continuity_lemma1} ensures that
\begin{equation}\label{1221}
f^{(n)}_k(t^{(n)}_k+\cdot)\to f_k(t_k+\cdot),\;\;n\to \infty
\end{equation}
in $(D(\mr),d)$ for $k=-r_-, \ldots, r_+-1$.

Now assume that $u_1,\ldots,u_l$ are continuity points of $f_k(t_k+\cdot)$ for all $k\in\mz$.
We show that then
\begin{equation}    \label{eq:continuity}
\phi^{(l)}_c(m_n,(f^{(n)}_k)_{k\in\mz}) \to
\phi^{(l)}_c(m,(f_k)_{k\in\mz}),\quad n \to \infty.
\end{equation}
Indeed, in the given situation, \eqref{1221} implies that
$$  \big(f^{(n)}_k(t^{(n)}_k+u_1),\ldots,f^{(n)}_k(t^{(n)}_k+u_l)\big)
~\to~   \big(f_k(t_k+u_1),\ldots, f_k(t_k+u_l)\big),\;\;n\to
\infty $$ for $k=-r_-, \ldots, r_+-1$. Summation of these
relations over $k=-r_-, \ldots, r_+-1$ yields
\eqref{eq:continuity}.

Theorem 4.1 in \cite{Whitt:1980} tells us that addition on
$D(\mr)\times D(\mr)$ is continuous at those $(x,y)$ for which ${\rm
Disc}(x)\cap {\rm Disc}(y)=\varnothing$. Since this immediately
extends to any finite number of summands we conclude that
relations \eqref{1221} entail $$
\phi_c(m_n,(f^{(n)}_k)_{k\in\mz})=\sum_{k=-r_-}^{r_+-1}
f_k^{(n)}(t_k^{(n)}+\cdot) \ \to \ \sum_{k=-r_-}^{r_+-1}
f_k(t_k+\cdot)=\phi_c(m,(f_k)_{k\in\mz}),\;\;n\to\infty$$ in
$(D(\mr),d)$ provided that ${\rm Disc}(f_k(t_k+\cdot))\cap {\rm
Disc}(f_j(t_j+\cdot))=\varnothing$ for $k\neq j$.
\end{proof}


\begin{proof}[Proof of Theorem \ref{main1}]
We start by showing that the Lebesgue integrability of $G(t)=\me
[|X(t)|\wedge 1]$ ensures $|Y^\ast(u)|<\infty$ a.s. for each $u
\in \mr$. To this end, fix $u\in\mr$ and set
$ \mathcal{Z}_k:=X_{k+1}(u+S_k^\ast)\1_{\{S_k^\ast\geq -u\}}$, $k\in\mz$. We
infer
\begin{eqnarray}
\sum_{k \in \mz} \me [|\mathcal{Z}_k|\wedge 1] & = &
\sum_{k \in \mz} \me [(|X_{k+1}(u+S_k^\ast)|\wedge 1)\1_{\{S_k^\ast\geq-u\}}]    \notag \\
& = & \sum_{k \in \mz} \me [G(u+S_k^\ast)\1_{\{S_k^\ast\geq -u\}}]
~=~ \frac1\mu \int_0^\infty G(x) \, {\rm d}x ~<~ \infty
\label{eq:sum of Z_k}
\end{eqnarray}
having utilized \eqref{eq:intensity=Lebesgue} for the last
equality. Therefore $\sum_{k\geq 0}| \mathcal{Z}_k| < \infty$ a.s.~by the
two-series theorem which implies $|Y^\ast(u)|<\infty$ a.s.

Next, we prove that direct Riemann integrability of
$H_{\varepsilon}(t)=\me [\sup_{u\in[t,\,t+\varepsilon]}
|X(u)|\wedge 1]$ for some $\varepsilon > 0$ implies that $Y^\ast$
takes values in $D(\mr)$ a.s. Since locally uniform limits of elements
from $D(\mr)$ are again in $D(\mr)$, it suffices to check that
$Y^\ast(u):=\sum_{k\in\mz}X_{k+1}(u+S_k^\ast)$ converges uniformly
on every compact interval a.s. To this end, fix $a,b \in \mr$,
$a<b$. It suffices to consider the case when $a,b \in \mz$
(otherwise, replace $a$ by $\lfloor a \rfloor$ and $b$ by $\lceil
b \rceil$). Now notice that
\begin{equation*}
\sup_{u \in [a,\,b]} \sum_{k\in\mz} (|X_{k+1}(u+S_k^\ast)| \wedge
1)  ~\leq~  \sum_{k\in\mz} \sup_{u \in [a,\,b]}
(|X_{k+1}(u+S_k^\ast)| \wedge 1)
\end{equation*}
and that the series on the right-hand side of this equation is
measurable since the $X_{k+1}$, $k \in \mz$ take values in $D(\mr)$.
Further, we observe that
\begin{eqnarray}
\me\bigg[\sum_{k\in\mz} \sup_{u \in [a,\,b]}
(|X_{k+1}(u+S_k^\ast)| \wedge 1)\bigg]
& = &
\me\bigg[ \sum_{k\in\mz}  \me\bigg[\sup_{u \in [a,\,b]} (|X(u+S_k^\ast)| \wedge 1) \,\Big|\, (S^\ast_j)_{j \in \mz}\bigg]\bigg]   \notag  \\
& = &
\frac{1}{\mu} \int_{\mr} \me\bigg[\sup_{u \in [a,\,b]} (|X(u+t)| \wedge 1) \bigg]\, {\rm d}t  \notag  \\
& \leq &
\frac{1}{\mu} \sum_{k\in\mz}\sup_{t\in[k,\,k+1)} \me \bigg[\sup_{u\in[a+t,\,b+t]}|X(u)|\wedge 1\bigg]   \label{eq:sup_k sup_a,b 1st}
\end{eqnarray}
where the last equality is a consequence of
\eqref{eq:intensity=Lebesgue}. To check that the last series
converges, put $r := b-a$ and notice that
\begin{eqnarray}
\sum_{k\in\mz}\sup_{t\in[k,\,k+1)} \me \bigg[\sup_{u\in[a+t,\,b+t]}|X(u)|\wedge 1\bigg]
& \leq &
\sum_{k\in\mz} \me \bigg[\sup_{u\in[a+k,\,b+k+1)}|X(u)|\wedge 1\bigg]       \notag  \\
& \leq &
\sum_{j=0}^r \sum_{k\in\mz} \me \bigg[\sup_{u\in[a+k+j,\,a+k+j+1)}|X(u)|\wedge 1\bigg]  \notag  \\
& = &
\sum_{j=0}^r \sum_{k\in\mz} \me \bigg[\sup_{u\in[k,\,k+1)}|X(u)|\wedge 1\bigg]  \notag  \\
& = &
(r+1) \sum_{k \geq 0} \me \bigg[\sup_{u\in[k,\,k+1)}|X(u)|\wedge 1\bigg]        \label{eq:sup_k sup_a,b}
~<~ \infty
\end{eqnarray}
where the last equality follows from the fact that $X(u)=0$ for
$u<0$, while the finiteness of the last series is secured by
\eqref{dri_path_eq}. Thus, $\sum_{k\in\mz} (|X_{k+1}(u+S_k^\ast)|
\wedge 1)$ converges uniformly on $[a,b]$ a.s. Since the set of
$a,b \in \mz$ with $a<b$ is countable, $Y^\ast$ is indeed
$D(\mr)$-valued a.s.

Using Lemma \ref{import} and recalling that the space
$M_p(\mr)\times D(\mr)^\mz$ is separable we infer
\begin{equation}\label{point_process_conv_2}
\bigg(\sum_{k\geq 0}\delta_{t-S_k},(X_{k+1})_{k\in\mz}\bigg) \
\Rightarrow \
\bigg(\sum_{k\in\mz}\delta_{S^{\ast}_k},(X_{k+1})_{k\in\mz}\bigg), \
\ t\to\infty
\end{equation}
on $M_p(\mr)\times D(\mr)^\mz$ by Theorem 3.2 in
\cite{Billingsley:1968}.

\noindent {\it Proof of \eqref{main_relation_func}}. We shall use
Lemma \ref{continuity_lemma2}. To this end, observe that each
$S^{\ast}_j$ has an absolutely continuous distribution. In
particular, $\sum_{j \in \mz} \delta_{S^{\ast}_j}(\{-c,0,c\}) = 0$
a.s.\ for every $c>0$. Further, for $i<j$, let $D_{i+1}:={\rm
Disc}(X_{i+1})$, $D_{j+1}:={\rm Disc}(X_{j+1})$. For a set $A
\subset \mr$ and $b \in \mr$, we write $A-b$ for the set $\{a-b: a
\in A\}$. With this notation, we have
\begin{eqnarray*}
&&\hspace{-3cm}\mmp\big\{ {\rm Disc}\big(X_{i+1}(S^\ast_i+\cdot)\big) \cap {\rm Disc}\big(X_{j+1}(S^\ast_j+\cdot)\big)\neq \varnothing \big\}\\
&&= \mmp\big\{(D_{i+1}-S^\ast_i) \cap (D_{j+1}-S^\ast_j) \neq \varnothing \big\}  \\
&&\leq \mmp\big\{\big((D_{i+1}-S^\ast_i) \setminus (\mathcal{D}_{X}-S^\ast_i)\big) \cap (D_{j+1}-S^\ast_j) \neq \varnothing \big\}   \\
&&+\mmp\big\{(\mathcal{D}_{X}-S^\ast_i) \cap \big((D_{j+1}-S^\ast_j) \setminus
(\mathcal{D}_{X}-S^\ast_j)\big) \neq \varnothing \big\}\\
&&+\mmp\big\{(\mathcal{D}_{X}-S^\ast_i) \cap (\mathcal{D}_{X}-S^\ast_j) \neq \varnothing \big\}.
\end{eqnarray*}
We now argue that the last three summands vanish.
As to the first, using the independence of $(X_{i+1},X_{j+1})$ and $(S^\ast_i,S^\ast_j)$ and conditioning with respect to $(S^\ast_i,S^\ast_j)$
we conclude that it suffices to show that $\mmp\big\{\big((D_{i+1}-s^\ast_i) \setminus (\mathcal{D}_{X}-s^\ast_i)\big) \cap (D_{j+1}-s^\ast_j) \neq \varnothing \big\}=0$
for fixed $s^\ast_i,s^\ast_j \in \mr$, $s^\ast_i < s^\ast_j$.
Since $X_{i+1}(s^\ast_i+\cdot)$ and $X_{j+1}(s^\ast_j+\cdot)$ are independent, we can argue as in the proof of Lemma 4.3 in \cite{Whitt:1980}:
\begin{align*}
\mmp\big\{\big((D_{i+1}-s^\ast_i) \setminus
(\mathcal{D}_{X}-s^\ast_i)\big)&\cap (D_{j+1}-s^\ast_j) \neq
\varnothing \big\}\\
&=   \int \mmp\big\{X_{i+1}(s^\ast_i+\cdot) \in
A(y) \big\} \,  \mmp\{X(s^\ast_j + \cdot) \in {\rm d}y\}
\end{align*}
where $A(y) = \big\{x \in D(\mr): \big({\rm Disc}(x)\setminus (\mathcal{D}_{X}-s^\ast_i)\big)\cap {\rm Disc}(y) \not = \varnothing\big\}.$
For any $y \in D(\mr)$,
\begin{equation*}
\mmp\big\{X_{i+1}(s^\ast_i+\cdot) \in A(y) \big\}
= \sum_{t \in {\rm Disc}(y) \setminus (\mathcal{D}_{X}-s^\ast_i)} \mmp\big\{X_{i+1}(s^\ast_i+t) \neq X_{i+1}((s^\ast_i+t)-) \big\} = 0.
\end{equation*}
The second term can be treated similarly. As to the third term,
$\mmp\{(\mathcal{D}_{X}-S^\ast_i) \cap (\mathcal{D}_{X}-S^\ast_j)
\neq \varnothing \} \leq
\mmp\{S_j^{\ast}-S_i^{\ast}\in\Delta_X\}$. If $i \geq 0$ or $j<0$,
then $S^\ast_j-S^{\ast}_i \od S_{j-i}$. If $i<0\leq j$, then
$S^\ast_i-S^{\ast}_j \od \xi_0+S_{i-j-1}$. Since the sets of atoms
of distributions of $\xi_0$ and $\xi$ are the same, we conclude
that $\mmp\{S^\ast_j-S^{\ast}_i\in \Delta_X\}=0$ by
\eqref{no_common_disc_ass}. Consequently, $\mmp\{{\rm
Disc}(X_{j+1}(S^\ast_j+\cdot))\cap {\rm
Disc}(X_{i+1}(S^\ast_i+\cdot))\neq \varnothing\} = 0$ for $i\neq
j$. This justifies using Lemma \ref{continuity_lemma2}, according
to which $\phi_c$ is a.s.~continuous at
$\big(\sum_{k\in\mz}\delta_{S^\ast_k},(X_{k+1})_{k\in\mz}\big)$.
Applying now the continuous mapping theorem to
\eqref{point_process_conv_2} yields
\begin{eqnarray}
Y_c(t,u)
& := &
\sum_{k\geq 0} X_{k+1}(u+t-S_k)\1_{\{|t-S_k|\leq c\}}   \notag  \\
& \od & \phi_c\bigg(\sum_{k\geq
0}\delta_{t-S_k},(X_{k+1})_{k\in\mz}\bigg) ~\Rightarrow~
\phi_c \bigg(\sum_{k\in\mz}\delta_{S^{\ast}_k},(X_{k+1})_{k\in\mz}\bigg)   \notag  \\
& = & \sum_{k \in\mz}X_{k+1}(u+S^\ast_k)\1_{\{|S^\ast_k| \leq c\}}
=: Y^\ast_c(u), \ \ t\to\infty  \label{func_conv1}
\end{eqnarray}
in $(D(\mr),d)$. Here, it should be noticed that one equality in
the relation above is distributional rather than pathwise.
However, since $(X_{k+1})_{k \in \mz}$ is an i.i.d.\ sequence, its
distribution is invariant under permutations. And, what is more,
due to the independence between the sequences $(\xi_k)_{k \in
\mn}$ and $(X_k)_{k \in \mz}$, the distribution of $(X_{k+1})_{k
\in \mz}$ is even invariant under $(\xi_k)_{k \in \mn}$-measurable
permutations.

Using Proposition \ref{conv_in_d} and following the reasoning
in the proof of Proposition 4.18 in \cite{Resnick:1987} we
conclude that in order to prove \eqref{main_relation_func} it
suffices to check that
\begin{equation}\label{main_relation_func12}
Y(t+u) \ \Rightarrow \ Y^\ast(u),\quad t\to\infty
\end{equation}
in $(D[a,b],d_0^{a,b})$ for any $a$ and $b$, $a<b$ which are
not fixed discontinuities of $Y^\ast$. To this end, first
observe that \eqref{func_conv1} implies
\begin{equation}\label{func_conv12}
Y_c(t,\cdot) \ \Rightarrow \ Y^\ast_c(\cdot),\quad t\to\infty
\end{equation}
in $(D[a,b],d_0^{a,b})$ for any $a$ and $b$, $a<b$ which are not
fixed discontinuities of $Y^\ast_c$. It can be checked that for
each fixed $c>0$ the set of fixed discontinuities of $Y^\ast_c$ is
a subset of the set of fixed discontinuities of $Y^\ast$. Since
$Y^\ast$ has paths in $D(\mr)$ and is stationary, the set of fixed
discontinuities of $Y^\ast$ is empty. Hence \eqref{func_conv12}
holds for any $a$ and $b$, $a<b$. Now \eqref{main_relation_func12}
follows from Theorem 4.2 in \cite{Billingsley:1968} if we can
prove that
\begin{equation}\label{ya_conv} Y^\ast_c \ \Rightarrow \
Y^\ast,\quad c\to\infty
\end{equation}
in $(D[a,b],d_0^{a,b})$ and that
\begin{equation}\label{bill_cond1}
\lim_{c\to\infty}\limsup_{t\to\infty}
\mmp\bigg\{d_0^{a,b}(Y_c(t,\cdot),Y(t+\cdot))>\varepsilon\bigg\}=0
\end{equation}
for all $\varepsilon>0$ and any $a,b\in\mr$, $a<b$.

\noindent {\it Proof of \eqref{bill_cond1}}. Since $d_0^{a,b}$ is
dominated by the uniform metric on $[a,b]$ it suffices to prove
that
\begin{equation}    \label{bill_cond}
\lim_{c\to\infty}\limsup_{t\to\infty}
\mmp\bigg\{\sup_{u\in[a,\,b]}\bigg|\sum_{k\geq 0}
X_{k+1}(u+t-S_k)\1_{\{|t-S_k| >c\}} \bigg|>\varepsilon\bigg\}=0
\end{equation}
for all $\varepsilon>0$ and any $a,b\in\mr$, $a<b$.

To prove \eqref{bill_cond}, set
$M_k(t):=\sup_{u\in[a,\,b]}|X_k(u+t)|$, $k\in\mz$ and $K(t):=\me
[M_{1}(t)\wedge 1]$ and write
\begin{align*}
\mmp\bigg\{ & \sup_{u\in[a,\,b]}\bigg|\sum_{k\geq 0} X_{k+1}(u+t-S_k)\1_{\{|t-S_k| > c\}}\bigg|>\varepsilon\bigg\}  \\
&\leq \mmp\bigg\{\sum_{k\geq 0} M_{k+1}(t-S_k)\1_{\{|t-S_k| > c\}}>\varepsilon\bigg\}   \\
&\leq \mmp\bigg\{\sum_{k\geq 0}M_{k+1}(t-S_k)\1_{\{|t-S_k| > c,\,M_{k+1}(t-S_k)\leq 1\}}>\varepsilon/2\bigg\}\\
&\hphantom{\leq} + \mmp\bigg\{\sum_{k\geq 0}M_{k+1}(t-S_k)\1_{\{|t-S_k| > c,\,M_{k+1}(t-S_k) > 1\}}>\varepsilon/2\bigg\}\\
&\leq \frac{2}{\varepsilon}\me \bigg[ \sum_{k \geq 0}K(t-S_k)\1_{\{|t-S_k| > c\}}\bigg] + \sum_{k\geq 0}\mmp\big\{|t-S_k|>c,\,M_{k+1}(t-S_k)>1\big\}
\end{align*}
where the last inequality follows from Markov's inequality.
\eqref{eq:sup_k sup_a,b} implies
\begin{equation*}
\sum_{k \in \mz} \sup_{t \in [k,\,k+1)} K(t)  ~<~ \infty.
\end{equation*}
This together with the local Riemann integrability of $K$
imply that $K$ is dRi on $\mr$.
Consequently,
$$  \lit \me\bigg[\sum_{k\geq 0}K(t-S_k)\1_{\{|t-S_k| > c\}}\bigg] = \mu^{-1}\int_{\{|t|>c\}}K(t) \, {\rm d}t   $$
by the key renewal theorem. Observe also that the last expression tends to zero as $c\to\infty$.
Further, the function $t \mapsto \mmp\{M_1(t)>1\}$ is dRi on $\mr$, for
$$  \mmp\{M_1(t)>1\}\leq \me [M_1(t)\wedge 1] = K(t).   $$
Again from the key renewal theorem, we conclude that
$$  \lit \sum_{k\geq 0}\mmp\big\{|t-S_k|>c,\,M_{k+1}(t-S_k)>1\big\} = \frac{1}{\mu} \int_{\{|x|>c\}}\mmp\{M_1(x)>1\} \, {\rm d}x.   $$
The last expression tends to $0$ as $c \to \infty$, which
establishes \eqref{bill_cond}.

\noindent {\it Proof of \eqref{ya_conv}}. In fact, we
claim that even the stronger statement $Y^\ast_c  \to Y^\ast$ as
$c\to\infty$ in $(D(\mr),d)$ a.s.\ holds. To prove this, we fix
arbitrary $a,b\in\mz$, $a<b$ and observe that it is sufficient to
check that the right-hand side of
\begin{eqnarray*}
\sup_{u\in [a,b]} |Y^\ast_c(u) - Y^\ast(u)|
& = &
\sup_{u\in [a,b]} \bigg|\sum_{k\in\mz}X_{k+1}(u+S^{\ast}_k)\1_{\{|S^{\ast}_k|>c\}}\bigg|\\
&\leq&
\sum_{k\in\mz} \sup_{u \in [a,b]} |X_{k+1}(u+S_k^\ast)| \1_{\{|S^{\ast}_k|>c\}}
\end{eqnarray*}
tends to zero as $c\to\infty$ a.s. To this end, notice that
\begin{eqnarray*}
\sum_{k\in\mz} \sup_{u \in [a,b]} (|X_{k+1}(u+S_k^\ast)| \wedge 1)
\end{eqnarray*}
has finite expectation by \eqref{eq:sup_k sup_a,b 1st} and
\eqref{eq:sup_k sup_a,b} and, hence, $\sum_{k\in\mz} \sup_{u \in
[a,b]} |X_{k+1}(u+S_k^\ast)|<\infty$ a.s. Therefore,
\begin{equation*}
\sup_{u\in [a,b]} |Y^\ast_c(u) - Y^\ast(u)| ~\leq~  \sum_{k\in\mz}
\sup_{u \in [a,b]} |X_{k+1}(u+S_k^\ast)| \1_{\{|S^{\ast}_k|>c\}}
~\to~   0
\end{equation*}
as $c\to\infty$ a.s.\ by the monotone (or dominated) convergence
theorem. The asserted a.s.\ convergence of $Y^\ast_c$ to $Y^\ast$
as $c\to\infty$ in $(D(\mr),d)$ follows.

\noindent {\it Proof of \eqref{main_relation}}.
Fix $l\in\mn$ and real numbers $\alpha_1,\ldots,\alpha_l$ and $u_1,\ldots,u_l$. For $k\in\mz$,
the number of jumps of $X_{k+1}$ is at most countable a.s.
Since the distribution of $S_k^\ast$ is absolutely continuous, and
$S_k^\ast$ and $X_{k+1}$ are independent we infer
$$  \mmp\{S_k^\ast+u \in {\rm Disc}(X_{k+1})\}=0$$
for any $u\in\mr$. According to Lemma \ref{continuity_lemma2}, for
every $c>0$, the mapping $\phi^{(l)}_c$ is a.s.~continuous at
$\big(\sum_{k\in\mz}\delta_{S^\ast_k},(X_{k+1})_{k\in\mz}\big)$.
Now apply the continuous mapping theorem to
\eqref{point_process_conv_2} twice (first using the map
$\phi^{(l)}_c$ and then the map $(x_1,\ldots,x_l) \mapsto \alpha_1
x_1 + \ldots + \alpha_l x_l$) to obtain that
\begin{equation*}
\sum_{i=1}^l \alpha_i Y_c(t,u_i) \ \dod \
\sum_{i=1}^l\alpha_iY_c^\ast(u_i), \ \ t\to\infty.
\end{equation*}
The proof of \eqref{main_relation} is complete if we verify
\begin{equation}\label{ya_conv_3}
\sum_{i=1}^l \alpha_i Y^\ast_c(u_i) \ \dod \ \sum_{i=1}^l \alpha_i Y^\ast(u_i),\;\;c\to\infty
\end{equation}
and
\begin{equation}\label{bill_cond_4}
\lim_{c\to\infty} \limsup_{t\to\infty} \mmp\bigg\{\bigg|\sum_{i=1}^l \alpha_i \sum_{k\geq 0}X_{k+1}(u_i+t-S_k)\1_{\{|t-S_k| > c\}}\bigg|>\varepsilon\bigg\}=0
\end{equation}
for all $\varepsilon>0$. As to \eqref{ya_conv_3}, we claim
that the stronger statement $Y^\ast_c(u) \to Y^\ast(u)$ as $c \to
\infty$ a.s.\ for all $u \in \mr$ holds. Indeed, as we have shown
in \eqref{eq:sum of Z_k},
\begin{eqnarray*}
\me\bigg[\sum_{k \in \mz} |X_{k+1}(u+S_k^\ast)| \wedge 1\bigg]  ~<~ \infty,
\end{eqnarray*}
in particular, $\sum_{k \in \mz} |X_{k+1}(u+S_k^\ast)| < \infty$
a.s. Hence, by the monotone (or dominated) convergence theorem,
\begin{equation*}
|Y^\ast_c(u) - Y^\ast(u)| ~\leq~  \sum_{k \in \mz}
|X_{k+1}(u+S_k^\ast)| \1_{\{|S_k^{\ast}|>c\}}  ~\to~   0
\end{equation*}
as $c \to \infty$ a.s.
Further, \eqref{bill_cond_4} is a consequence of
$$  \lim_{c\to\infty}\limsup_{t\to\infty}\mmp\bigg\{\Big|\sum_{k\geq 0}X_{k+1}(u+t-S_k)\1_{\{|t-S_k| > c\}}\Big|>\varepsilon\bigg\} = 0 $$
for every $u \in \mr$. Write
\begin{align*}
\mmp\bigg\{ & \bigg|\sum_{k\geq 0} X_{k+1}(u+t-S_k)\1_{\{|t-S_k| > c\}}\bigg| > \varepsilon\bigg\}  \\
&\leq\mmp\bigg\{\sum_{k\geq 0}
|X_{k+1}(u+t-S_k)|\1_{\{|t-S_k| > c\}}>\varepsilon\bigg\}\\
&\leq \mmp\bigg\{\sum_{k\geq 0}|X_{k+1}(u+t-S_k)|\1_{\{|t-S_k| > c,\,|X_{k+1}(u+t-S_k)|\leq 1\}}>\varepsilon/2\bigg\}\\
&\hphantom{\leq} +\mmp\bigg\{\sum_{k\geq 0}|X_{k+1}(u+t-S_k)|\1_{\{|t-S_k| > c,\,|X_{k+1}(u+t-S_k)| > 1\}}>\varepsilon/2\bigg\}\\
&\leq \frac{2}{\varepsilon} \me \bigg[\sum_{k\geq 0}G(u+t-S_k)\1_{\{|t-S_k| > c\}}\bigg] + \sum_{k\geq 0}\mmp\big\{|t-S_k|>c,\,|X_{k+1}(u+t-S_k)|>1\big\}
\end{align*}
and observe that since $G$ is dRi on $[0,\infty)$, so is $t\mapsto
G(u+t)$ on $\mr$ whence
$$  \underset{c\to\infty}{\lim}\,\lit \me \bigg[\sum_{k\geq 0} G(u+t-S_k)\1_{\{|t-S_k| > c\}}\bigg]
= \frac{1}{\mu} \underset{c\to\infty}{\lim}\,\int_{\{|x|>c\}}G(u+x) \, {\rm d}x=0.  $$
In view of
$$  \mmp\{|X(u+t)|>1\}\leq G(u+t),  $$
the function $t\mapsto \mmp\{|X(u+t)|>1\}$ is dRi on $\mr$, which gives
\begin{eqnarray*}
&&\hspace{-2cm}  \underset{c\to\infty}{\lim}\,\lit \sum_{k\geq 0}\mmp\big\{|t-S_k|>c,\,|X_{k+1}(u+t-S_k)|>1\big\}\\
&&\hspace{2cm}=\frac{1}{\mu}\underset{c\to\infty}{\lim}\,\int_{\{|x|>c\}}\mmp\{|X(u+x)|>1\} \, {\rm d}x=0.
\end{eqnarray*}
This finishes the proof of \eqref{bill_cond_4}.
\eqref{ya_conv_3} can be checked along the same lines. We omit the details.
\end{proof}

\section*{Acknowledgements}
We express
our sincere gratitude to two anonymous referees who prepared
reports full of valuable feedback. Their comments greatly helped
improving the quality of our paper and its presentation.
\normalsize

\bibliography{IksMarMei_2015_2_Bib}

\end{document}